\documentclass[12pt,reqno]{amsart}

\usepackage[a4paper,bindingoffset=0.1in,left=1.2in,right=1.3in,top=1.4in,bottom=1.5in,footskip=.25in]{geometry}

\usepackage{indentfirst,amssymb,amsmath,amsthm}  
\usepackage{amscd, amsfonts, graphicx,tikz, color, environ}
\usepackage{newtxtext,newtxmath}
\usepackage{setspace}
\usepackage{times}

\usepackage{hyperref}
\hypersetup{colorlinks=true, linkcolor=blue,citecolor=blue, urlcolor=blue}
\usepackage{mathtools}
\DeclarePairedDelimiter\floor{\lfloor}{\rfloor}

\DeclareMathOperator{\ann}{ann}
\DeclareMathOperator{\ord}{ord}
\DeclareMathOperator{\dime}{dim}
\DeclareMathOperator{\pideg}{PI-deg}

\DeclareMathOperator{\gcdi}{gcd}
\DeclareMathOperator{\spect}{spec}

\numberwithin{equation}{section}

\newtheorem{theo}{Theorem}[section]

\newtheorem{rema}{Remark}[section]
\newtheorem{coro}{Corollary}[section]
\newtheorem{prop}{Proposition}[section]
\newtheorem{Notation}{Notation}[section]
\newtheorem*{theorem A}{Theorem A}
\newtheorem*{theorem B}{Theorem B}

\makeatletter
\g@addto@macro{\endabstract}{\@setabstract}
\newcommand{\authorfootnotes}{\renewcommand\thefootnote{\@fnsymbol\c@footnote}}%
\makeatother

\begin{document}

\label{'ubf'}  
\setcounter{page}{1} 
\baselineskip .65cm 
\pagenumbering{arabic}

\markboth 
{\hspace*{-9mm} \centerline{\footnotesize
         Multiparameter Quantum Affine Space  }
                 }
{ \centerline                           {\footnotesize \ 
                 Snehashis Mukherjee and Sanu Bera } 
\hspace*{-9mm}}

\begin{center}
  \Large
{\textbf{ Construction of Simple Modules over the Quantum Affine Space}}
\par \vspace{.5cm}

  \normalsize
  \authorfootnotes
  {\bf Snehashis Mukherjee}\footnote{tutunsnehashis@gmail.com}
   \hspace{.6cm}
  {\bf Sanu Bera}\footnote{sanubera6575@gmail.com}
 \par
  School of Mathematical Sciences, \par Ramakrishna Mission Vivekananda Educational and Research Institute, 
  \par Belur Math, Howrah, West Bengal, Box: 711202, India.
\end{center}

\hspace{.8cm}

\noindent{\small{\bf Abstract:}
The coordinate ring $\mathcal{O}_{\mathbf{q}}(\mathbb{K}^n)$ of quantum affine space is the $\mathbb{K}$-algebra presented by generators $x_1,\cdots ,x_n$ and relations $x_ix_j=q_{ij}x_jx_i$ for all $i,j$.  We construct simple $\mathcal{O}_{\mathbf{q}}(\mathbb{K}^n)$-modules in a more general setting where the entries $q_{ij}$ lie in a torsion subgroup of $\mathbb{K}^*$ and show analogous results hold as in \cite{mo}.
}

\section{\textbf{Introduction}}
Quantum affine spaces, the well-known deformations of commutative polynomials rings, are the most elementary noncommutative algebras occurring among quantized coordinate rings. Let $\mathbb{K}$ be a field and $\mathbb{K}^*$ denote the group $\mathbb{K}\setminus\{0\}$. A square matrix $\mathbf{q}\in \mathcal{M}_n(\mathbb{K}^{*})$ is multiplicatively antisymmetric if $q_{ii}=1$ for all $i$ and $q_{ji}={q_{ij}}^{-1}$ for all $i\neq j$. Given such a matrix, the corresponding multiparameter coordinate ring of quantum affine $n$-space, or just multiparameter quantum affine $n$-space, is the $\mathbb{K}$-algebra $\mathcal{O}_{\mathbf{q}}(\mathbb{K}^n)$ generated by the variables $x_1,\cdots ,x_n$ subject only to the relations
\begin{equation} \label{relation}
x_ix_j=q_{ij}x_jx_i, \ \ \ \forall\ \ \ 1 \leq i,j\leq n.
\end{equation}
First of all, the multiparameter quantum affine spaces $\mathcal{O}_{\mathbf{q}}(\mathbb{K}^n)$ are iterated skew polynomial rings, in which each iteration is twisted only by an automorphism. Hence $\mathcal{O}_{\mathbf{q}}(\mathbb{K}^n)$ is an affine noetherian domain (cf. \cite{mr}, 2.9) and has a $\mathbb{K}$ basis consisting of $\{x_1^{r_1}\cdots x_n^{r_n}~|~r_i\geq 0\}$. Also it is clear that $\mathcal{O}_{\mathbf{q}}(\mathbb{K}^n)$ is a constructible $\mathbb{K}$-algebra and so it satisfies the Nullstellensatz over $\mathbb{K}$; in particular, $\mathcal{O}_{\mathbf{q}}(\mathbb{K}^n)$ is a Jacobson ring (cf. \cite{mr}, 9.4.21). Furthermore, standard methods show that $\mathcal{O}_{\mathbf{q}}(\mathbb{K}^n)$ is Auslander regular and Cohen-Macaulay, with global and Gelfand-Kirillov dimension $n$ (cf. \cite{gl}, 2.5). Also, the previous results can be used to show that $\spect \mathcal{O}_{\mathbf{q}}(\mathbb{K}^n)$ is catenary (\cite{gl}, 2.6). The prime and primitive spectra of the algebra $\mathcal{O}_{\mathbf{q}}(\mathbb{K}^n)$ were studied  for arbitrary parameter matrices $\mathbf{q}$ in \cite{bg1},\cite{ckp},\cite{glt}. Also the isomorphism problem for the quantum affine spaces was solved in \cite{jg}.
\par The coordinate ring $\mathcal{O}_{\mathbf{q}}(\mathbb{K}^n)$ plays a fundamental role in non-commutative geometry (see \cite{bvf}). Suitable localizations of the quantum affine space arise in the representation theory of torsion-free
nilpotent groups (see \cite{cb}). Moreover, Such algebras form an essential part of the analysis of quantum groups at roots of unity (see \cite{cp}).
\par In  \cite{mo}, Kangju Min and Sei-Qwon Oh investigated the simple $\mathcal{O}_{\mathbf{q}}(\mathbb{C}^n)$-modules in the uniparameter case 
\[
q_{ij}=q,\ \ \ \forall \ \ \ 1\leq i<j \leq n,
\]
where $q$ is a primitive $m$-th root of unity and gave an explicit construction of the simple modules. In particular, they established the following fact: 
\begin{theo}[\cite{mo}, Theorem $5$] \label{th1}
There is a surjective map $\Phi$ from $\mathbb{C}^n$ onto the set of all simple $\mathcal{O}_{\mathbf{q}}(\mathbb{C}^n)$-modules in the case when $q$ is a primitive $m$-th root of unity such that 
\[\dime_\mathbb{C}\Phi\left(\underline{\alpha}\right) =m^{\floor*{\frac{p}{2}}},\]
where $p$ is the number of non zero $\alpha_i$ in $\underline{\alpha}=(\alpha_1,\cdots,\alpha_n)\in \mathbb{C}^n$ and $\floor*{x}$ is the integer part of $x$. 
\end{theo}
~\par  However, to date no full classification of irreducible representations of the algebra $\mathcal{O}_{\mathbf{q}}(\mathbb{K}^n)$ has appeared in the literature. In this article, we wish to generalize the above result for multiparameter case assuming that  $\mathbb{K}$ is an algebraically closed field and the multiplicative group $\mathbf{\Lambda}$ generated by the multiparameters $q_{ij}$ is a torsion (and hence cyclic) subgroup of $\mathbb{K}^*$. Let $m$ be the order and $q$ be a generator of the group $\mathbf{\Lambda}$. Then the multiparameters are of the form 
\begin{equation}\label{gen}
    q_{ij}=q^{h_{ij}},\ \ h_{ij} \in {\mathbb{Z}}, \ \ \ \forall \ \ \ 1\leq i,j \leq n.
\end{equation}
If $m=p^{l_1}_1\cdots p^{l_k}_k$ be the prime factorization of $m$, then the set 
\begin{equation}\label{mc}
    \mathbf{P}:=\{~\overline{p}^{j_1}_1\cdots \overline{p}^{j_k}_k~|~0\leq j_i\leq l_i,~ i=1,\cdots,k~\}\subseteq {\mathbb{Z}}/{m\mathbb{Z}}.
\end{equation}
is a multiplicatively closed set of representatives for the multiplicative cosets of the unit group  $({\mathbb{Z}}/{m\mathbb{Z}})^*$~ (see  \cite{nb}, Section III).
\par Throughout this paper a module means a right module and $\mathbb{K}$ denotes an algebraically closed field. Given a simple module $N$ over $\mathcal{O}_{\mathbf{q}}(\mathbb{K}^n)$ we may assume that the action of the variables $x_1,\cdots,x_n$ on $N$ is non trivial. Otherwise, if the action of a variable, say $x_j$ on $N$ is trivial, then $x_j\in \ann(N)$ and $N$ becomes a simple $\mathcal{O}_{\mathbf{q}'}(\mathbb{K}^{n-1})$-module,  where $\mathbf{q}'$ is the $(n-1)\times (n-1)$ multiplicatively antisymmetric matrix obtained form $\mathbf{q}$ by removing $j$-th row and $j$-th column simultaneously.
\medskip
~\par Let $\mathbb{A}_{\mathbf{q}}$ be the localization of $\mathcal{O}_{\mathbf{q}}(\mathbb{K}^n)$ with respect to the Ore set $\mathbf{X}$ generated by $x_1,\cdots,x_n$. Then $\mathbb{A}_{\mathbf{q}}$ is the $\mathbb{K}$-algebra $\mathcal{O}_{\mathbf{q}}\left(\mathbb{(K^*)}^n\right)$ generated by the variables $x_1,\cdots,x_n$ together with their inverses which satisfy the relations (\ref{relation}). This resulting algebra is of type investigated by J. C. McConnell and J. J. Pettit (cf. \cite{mp}). Also the prime and primitive spectra of this algebra were analyzed in \cite{glt}. In the literature, the ring $\mathbb{A}_{\mathbf{q}}$ is called a quantum torus of rank $n$. In particular, when $n=2$ we write $\mathbb{A}_{q}$ instead of $\mathbb{A}_{\mathbf{q}}$ where $xy=qyx$.
\medskip
\par In \cite{nb}, Karl-Hermann Neeb investigated the normal form of rational quantum tori, that is, assuming that the group $\mathbf{\Lambda}$ is torsion. In particular, the following result was shown:
\begin{theo}[\cite{nb}, Theorem III.4]\label{neeb}
For any rational quantum torus $\mathbb{A}_{\mathbf{q}}$ of rank $n$ over $\mathbb{K}$, there exists $s\in\mathbb{N}_0$ with $2s\leq n$ and $h_2\mid h_3\mid\cdots\mid h_s$ in $\mathbf{P}\setminus\{0\}$ such that
\begin{equation}
    \mathbb{A}_{\mathbf{q}} \cong  \mathbb{A}_q \otimes \mathbb{A}_{q^{h_2}} \otimes \mathbb{A}_{q^{h_3}}\otimes \cdots \otimes \mathbb{A}_{q^{h_s}}\otimes \mathbb{K}\left[\mathbb{Z}^{n-2s}\right],\ \ \  2s<n
\end{equation}
\begin{center}
   \text{or}
\end{center}
\begin{equation}\label{nn2}
\mathbb{A}_{\mathbf{q}} \cong  \mathbb{A}_q \otimes \mathbb{A}_{q^{h_2}} \otimes \mathbb{A}_{q^{h_3}}\otimes \cdots \otimes \mathbb{A}_{q^{h_{s-1}}}\otimes \mathbb{A}_{q^{zh_{s}}}~,\ \ \ \ \  2s = n
\end{equation}
for some $z \in \mathbb{N}$ with $\ord \left(q^{zh_{s}}\right) = \ord \left(q^{h_{s}}\right)$ and for quantum tori $\mathbb{A}_{q^{h_i}}$ (with $h_1=1$) of rank $2$.
\end{theo}

\begin{rema}\label{n2}
Since we are only interested in constructing simple $\mathbb{A}_{\mathbf{q}}$-modules upto isomorphism, in view of Theorem (\ref{inj}) it suffices to assume $h_s=zh_s$ in (\ref{nn2}). This allows us to work with the same normal form, viz  
\begin{equation}\label{final}
\mathbb{A}_{\mathbf{q}} \cong  \mathbb{A}_q \otimes \mathbb{A}_{q^{h_2}} \otimes \mathbb{A}_{q^{h_3}}\otimes \cdots \otimes \mathbb{A}_{q^{h_s}}\otimes \mathbb{K}\left[\mathbb{Z}^{n-2s}\right]
\end{equation}
in both cases of Theorem (\ref{neeb}). This normal form will play a central role in our module constructions. 
\end{rema}
\begin{rema}\label{limit}
The explicit description of the set $\mathbf{P}$ in (\ref{mc}), it follows that every $h_i\in \mathbf{P}\setminus\{0\}$ in the normal form (\ref{final}) must divides $m$.
We let 
\begin{equation}\label{pid}
k_i:=\frac{m}{h_i},\ \ \ \forall \ \ \ 1\leq i\leq s.
\end{equation}
\end{rema}
\begin{Notation}\label{not}
We let $X^{\pm1}_i$ and $X^{\pm1}_{i+s}$ be the generators of $\mathbb{A}_{q^{h_i}}$ with $h_1=1$, which satisfy the relation 
\begin{equation}\label{decom}
X_iX_{i+s}={q^{h_i}}X_{i+s}X_i,
\end{equation}
for each $1\leq i\leq s$. Also let the variables $X_{2s+1},\cdots,X_n$ generate the group algebra $\mathbb{K}\left[\mathbb{Z}^{n-2s}\right]$.
\end{Notation}
The following theorem generalizes the construction in \cite{mo} of simple modules for the more general 
relations (\ref{gen}) among the multiparameters.
\begin{theorem A}[Construction of Simple Modules]
For $ \underline\alpha:=(\alpha_1,\cdots,\alpha_n) \in \left(\mathbb{K}^*\right)^n$, let $M(\underline{\alpha})$ be the $\mathbb{K}$-vector space with basis $e(a_1,\cdots,a_s)$, where $0\leq a_i \leq k_i-1,~1\leq i\leq s$. Then there is an $\mathbb{A}_{\mathbf{q}}$-module structure on $M(\underline{\alpha})$ define as follows:
\begin{align*}
&e\left(a_1,\cdots,a_s\right)X_i=\alpha_i~e\left(a_1,\cdots,a_i+_i 1,\cdots,a_s\right),\ \ \ \forall \ \ \ 1 \leq i \leq s\\
&e\left(a_1,\cdots,a_s\right)X_{i+s}=\alpha_i^{-1}\alpha_{i+s}\left(q^{h_i}\right)^{a_i-1}e\left(a_1,\cdots,a_i+_i(-1),\cdots,a_s\right),\ \forall \ \ 1 \leq i \leq s\\
&e\left(a_1,\cdots,a_s\right)X_{2s+j}=\alpha_{2s+j} ~e\left(a_1,\cdots,a_s\right),\ \ \ \forall \ \ \ 1\leq j \leq n-2s
\end{align*}
where $+_i$ is addition in the additive group ${\mathbb{Z}}/{k_i\mathbb{Z}}$, and $s$ and $h_i$'s are as in the Theorem (\ref{neeb}). Moreover, $M(\underline{\alpha})$ is a simple $\mathbb{A}_{\mathbf{q}}$-module of dimension $\prod\limits_{i=1}^s k_i$.
\end{theorem A}
The following result also holds.
\begin{theorem B}
Each simple $\mathcal{O}_{\mathbf{q}}(\mathbb{K}^n)$-module has the form $M(\underline{\alpha})$, for some $\underline{\alpha} \in {(\mathbb{K}^*)}^n$. Furthermore, there is a surjective map from $(\mathbb{K}^{*})^n$ onto the set of all simple $\mathcal{O}_{\mathbf{q}}(\mathbb{K}^n)$-modules.
\end{theorem B}
\section{\textbf{Preliminaries}}
In this section we recall some known facts concerning the prime Polynomial Identity algebra as well as the $\mathbb{K}$-algebra $\mathcal{O}_{\mathbf{q}}(\mathbb{K}^n)$ that we shall be needing in the development of our results.

In the root of unity case (due to the assumptions on multiparameters), we can easily check that the quantum affine space $\mathcal{O}_{\mathbf{q}}(\mathbb{K}^n)$ is  a finitely generated module over a central subalgebra $\mathbb{K}[x^m_1,\cdots,x^m_n]$, with basis $\{x^{i_1}_1\cdots x^{i_n}_n~|~0\leq i_1,\cdots,i_n\leq m-1\}$. Hence it follows that $\mathcal{O}_{\mathbf{q}}(\mathbb{K}^n)$ is a (prime) PI algebra (cf. [\cite{mr}, Corollary 13.1.13]). Similar argument shows that the quantum torus $\mathbb{A_{\mathbf{q}}}$ is also a (prime) PI algebra in root of unity case. This sufficient condition on the multiparameters $q_{ij}$ for a quantum affine space $\mathcal{O}_{\mathbf{q}}(\mathbb{K}^n)$ and its localization $\mathbb{A_{\mathbf{q}}}$ to be PI is, in fact, necessary (see \cite{bg}, Proposition I.14.2). The following results provide a link between module structure over $\mathcal{O}_{\mathbf{q}}(\mathbb{K}^n)$ and over $\mathbb{A_{\mathbf{q}}}$ respectively:
\begin{prop}[\cite{bg}, Proposition III.1.1.] \label{finite}
Every simple $\mathcal{O}_{\mathbf{q}}(\mathbb{K}^n)$-module (respectively, simple $\mathbb{A_{\mathbf{q}}}$-module) is finite dimensional vector space over $\mathbb{K}$.
\end{prop}
\begin{prop} \label{divisible}
Let $N$ be a simple $\mathcal{O}_{\mathbf{q}}(\mathbb{K}^n)$-module. Then there exists a unique simple $\mathbb{A}_{\mathbf{q}}$-module structure on $N$ and such a structure is compatible with the $\mathcal{O}_{\mathbf{q}}(\mathbb{K}^n)$-module structure. 
\end{prop}
 \begin{proof}
Clearly $N$ is $\mathbf{X}$-torsion free, because we have assumed that the action of each variable on $N$ is non trivial. Since $N$ is simple, $Nx=N$ for all $x \in \mathbf{X}$. Thus $N$ is $\mathbf{X}$-divisible. Hence the assertion follows from  Proposition (10.11) in \cite{gw}.
\end{proof}

The following theorem provides one of the key techniques for calculating the PI degree of the algebras we are interested in.
\begin{prop}[\cite{cp}, Proposition 7.1]\label{quan}
Let $\mathbf{q}=\left(q_{ij}\right)$ be an  $n \times n$ multiplicatively antisymmetric matrix over $\mathbb{K}$.
 Suppose that $q_{ij}=q^{h_{ij}}$ for all $i,j$, where $q \in \mathbb{K}^*$ is a primitive $m$-th root of unity and the $h_{ij} \in \mathbb{Z}$. Let $h$ be the cardinality of the image of the homomorphism 
\[
    \mathbb{Z}^n \xrightarrow{(h_{ij})} \mathbb{Z}^n \xrightarrow{\pi} \left(\mathbb{Z}/m\mathbb{Z}\right)^n,
\]
where $\pi$ denotes the canonical epimorphism. Then \[\pideg(\mathcal{O}_{\mathbf{q}}(\mathbb{K}^n))=\pideg(\mathbb{A_{\mathbf{q}}})=\sqrt{h}.\]
\end{prop}
By above proposition (\ref{quan}), it follows in particular that the PI degree of a quantum torus of rank $2$ with commutation relation $xy=q^ryx$ is equals to $\frac{m}{\gcdi(r,m)}$.
In light of Remark(\ref{n2}), normal form of quantum torus reduces the algebra $\mathbb{A}_{\mathbf{q}}$ to a tensor product of quantum tori $\mathbb{A}_{q^{h_i}}$ of rank $2$ and a group algebra. Thus the $\pideg$ of $\mathbb{A}_{\mathbf{q}}$ can be simplified using relation (\ref{pid}) as follows:
\[\pideg(\mathbb{A}_{\mathbf{q}})=\prod\limits_{i=1}^s\pideg(\mathbb{A}_{q^{h_i}})=\prod\limits_{i=1}^s\frac{m}{\gcdi(h_i,m)}=\prod\limits_{i=1}^sk_i.\]
Now Kaplansky's Theorem (\cite{mr}, Theorem 13.3.8) has a striking consequence in case of a prime affine PI algebra over an algebraically closed field.
\begin{prop}[\cite{bg}, Theorem I.13.5]\label{sim}
Let $R$ be a prime affine PI algebra over an algebraically closed field $\mathbb{K}$ with $\pideg(R) =n$ and $V$ be a simple $R$-module. Then $V$ is a $\mathbb{K}$-vector space of dimension $t$, where $t \leq n$ and $R/\ann_R(V) \cong M_t(\mathbb{K})$.
\end{prop}

The above proposition yields the important link between the PI degree of a prime affine PI algebra over an algebraically closed field and its irreducible representations.
 Thus from proposition (\ref{quan}) and (\ref{sim}) along with the assumptions on multiparameters, it is quite clear that each simple $\mathcal{O}_{\mathbf{q}}(\mathbb{K}^n)$-module has dimension at most $\prod\limits_{i=1}^s k_i$, where $k_i$'s are as in (\ref{pid}).

\par The Artin-Procesi Theorem (\cite{mr}, Theorem $13.7.14$) provides a relationships between the Azumaya property and regular primes. In particular, if a prime algebra $R$ is an Azumaya algebra over its centre $C$, then each simple $R$-module has exactly dimension equal to $\pideg(R)$. With the aid of Artin-Procesi Theorem, it can be easily shown that $\mathcal{O}_{\mathbf{q}}(\mathbb{K}^n)$ is not Azumaya for any $n\geq 2$. On the other hand, quantum torus $\mathbb{A}_{\mathbf{q}}$ is an Azumaya algebra over its centre (cf. \cite{cp}, $7.2$). Since each simple $\mathcal{O}_{\mathbf{q}}(\mathbb{K}^n)$-module can be extended uniquely to a simple $\mathbb{A}_{\mathbf{q}}$-module (by Proposition (\ref{divisible})), therefore for any simple $\mathcal{O}_{\mathbf{q}}(\mathbb{K}^n)$-module $N$,
\[\dime_\mathbb{K}(N) = \pideg(\mathbb{A}_{\mathbf{q}}) = \pideg(\mathcal{O}_{\mathbf{q}}(\mathbb{K}^n))=\prod\limits_{i=1}^s k_i,\]
where $k_i$'s are as in (\ref{pid}).

\section{\textbf{Construction of a Simple Module}} 
In this section we wish to construct simple modules over the coordinate ring  $\mathcal{O}_{\mathbf{q}}(\mathbb{K}^n)$. By Proposition (\ref{divisible}), it suffices to do this for the rational quantum torus $\mathbb{A}_{
\mathbf{q}}$ of rank $n$. We follow the line of reasoning in \cite{mo} and our construction proceeds in the following steps.
\medskip\\
\textbf{Step~1:}~(The representation space)~ For $\underline{\alpha}=(\alpha_1,\cdots,\alpha_n)\in \left(\mathbb{K}^*\right)^n$, let us consider the $\mathbb{K}$-vector space $M(\underline{\alpha})$ with basis consisting of  \[\{e(a_1,\cdots,a_s)~:~0\leq a_i \leq k_i-1,~ 1\leq i\leq s\},\] where $k_i$ and $s$ are as defined in Introduction section.
\medskip\\
\textbf{Step~2:}~(Module structure)~ To define $\mathbb{A}_{\mathbf{q}}$-module action on $\mathbb{K}$-space $M(\underline{\alpha})$, we will use the normal form of rational quantum torus $\mathbb{A}_{\mathbf{q}}$. In the light of the  Remark (\ref{n2}) along with the Notation (\ref{not}) , let $X_1,\cdots,X_n$ be monomials for $\mathbb{A}_{\mathbf{q}}$ satisfying the relations (\ref{decom}). Let us define the $\mathbb{A}_{\mathbf{q}}$-module structure on the $\mathbb{K}$-space $M(\underline{\alpha})$ by the action of the variable $X_k$'s on the basis vectors as follows:
\begin{align*}
&e(a_1,\cdots,a_s)X_i=\alpha_i~e(a_1,\cdots,a_i+_i1,\cdots,a_s),\ \ \ \forall \ \ \ 1 \leq i \leq s\\
&e(a_1,\cdots,a_s)X_{i+s} =\alpha_i^{-1}\alpha_{i+s}\left(q^{h_i}\right)^{a_i-1}~e(a_1,\cdots,a_i+_i(-1),\cdots,a_s), \ \forall \ \ 1\leq i\leq s\\
&e(a_1,\cdots,a_s)X_{2s+j}=\alpha_{2s+j} ~e(a_1,\cdots,a_s),\ \ \ \forall \ \ \ 1\leq j \leq n-2s
\end{align*}
where $+_i$ is addition in the additive group ${\mathbb{Z}}/{k_i\mathbb{Z}}$.
\medskip\\
\textbf{Step~3:}~(Well-definedness)~ In order to establish the well-definedness of the above rules, we need to check the following relations for $1\leq i\leq s$ and $0\leq a_i\leq {k_i-1}$:
\begin{equation}\label{w1}
\begin{rcases}
  e(a_1,\cdots,a_s)X_iX_{i+s}=q^{h_i}~e(a_1,\cdots,a_s)X_{i+s}X_i&\\
  e(a_1,\cdots,a_s)X_iX_j=e(a_1,\cdots,a_s)X_jX_i, &\forall \  \ \ j \neq i+s\\
 e(a_1,\cdots,a_s)X_{i+s}X_j=e(a_1,\cdots,a_s)X_jX_{i+s},& \forall \ \ \ j \neq i\\
  e(a_1,\cdots,a_s)X_{k}X_{l}=e(a_1,\cdots,a_s)X_{l}X_{k}, &\forall \ \ \ 2s+1\leq k,l\leq n
  \end{rcases}
\end{equation}
The relations (\ref{w1}) can be easily verified.  With this we have the following.

\begin{theorem A}
The module $M(\underline{\alpha})$ is a simple $\mathbb{A}_{\mathbf{q}}$-module of dimension $\prod\limits_{i=1}^s k_i$.
\end{theorem A}
\begin{proof}
Let $P$ be a non-zero submodule of $M(\underline{\alpha})$. We claim that $P$ contains a basis vector of the form $e(a_1,\cdots,a_s)$. Indeed, any member $p\in P$ is a finite $\mathbb{K}$-linear combination of such vectors. i.e.,
\[
  p:=\sum_{\text{finite}} \lambda_k~e\left(a_1^{(k)},\cdots,a_s^{(k)}\right)  
\]
for some $\lambda_k\in \mathbb{K}$. Suppose there exist two non-zero coefficients, say, $\lambda_u,\lambda_v$.
We can choose the smallest index $r$ such that $a_r^{(u)}\neq a_r^{(v)}$ in ${\mathbb{Z}}/{k_r\mathbb{Z}}$. Now the vectors $e\left(a_1^{(u)},\cdots,a_s^{(u)}\right)$ and $e\left(a_1^{(v)},\cdots,a_s^{(v)}\right)$ are eigenvectors of $X_rX_{r+s}$ associated with the eigenvalues $\alpha_{r+s}\left(q^{h_r}\right)^{a_r^{(u)}}=\mu_u$ (say) and $\alpha_{r+s}\left(q^{h_r}\right)^{a_r^{(v)}}=\mu_v$ (say) respectively. We claim that $\mu_u\neq \mu_v$. Indeed, \[\mu_u=\mu_v \implies q^{h_r\left(a_r^{(u)}-a_r^{(v)}\right)}=1 \implies k_r\mid \left(a_r^{(u)}-a_r^{(v)}\right),\] which is a contradiction.
\par Now $pX_rX_{r+s}-\mu_up$ is a non zero element in $P$ of smaller length than $p$. Hence by induction it follows that every non zero submodule of $M(\underline{\alpha})$ contains a basis vector of the form $e(a_1,\cdots ,a_s)$. Thus $M(\underline{\alpha})$ is simple by the actions of $X_i,~~1\leq i \leq s$.
\end{proof}
\section{\textbf{Main Results}}
\begin{theorem B} \label{isomorphism}
Let $N$ be a simple $\mathcal{O}_{\mathbf{q}}(\mathbb{K}^n)$-module. Then $N$ is isomorphic to $M(\underline{\alpha})$ as $\mathbb{A}_{\mathbf{q}}$-module for some $\underline{\alpha} \in \left(\mathbb{K}^*\right)^n$. 
\end{theorem B}
\begin{proof}
A given simple $\mathcal{O}_{\mathbf{q}}(\mathbb{K}^n)$-module structure on $N$ can be extended uniquely to a simple $\mathbb{A}_{\mathbf{q}}$-module structure on $N$ by Proposition (\ref{divisible}). In the light of Remark (\ref{n2}) and Notation (\ref{not}) , let $X_1,\cdots,X_n$ be monomials for $\mathbb{A}_{\mathbf{q}}$ satisfying the relations (\ref{decom}). Since each of the monomials 
\begin{equation}\label{cev}
   X_i^{k_i},\ X_iX_{i+s},\  X_{2s+1},\ X_{2s+2},\cdots,X_n,\ \ \ \forall\ \ \ 1\leq i\leq s 
\end{equation}
commutes and $N$ is a finite dimensional vector space over $\mathbb{K}$ (by Proposition (\ref{finite})), there is a common eigenvector $v$ of monomials (\ref{cev}).
Put
\[
\begin{array}{ll}
vX_i^{k_i}=\xi_iv, &\forall \ \ \ 1\leq i\leq s\\
vX_iX_{i+s}=\alpha_{i+s}v,&\forall \ \ \ 1\leq i\leq s\\
vX_{2s+j}=\alpha_{2s+j}v, &\forall \ \ \ 1\leq j\leq n-2s
\end{array}
\]
Let $\alpha_i$ be a $k_i$-th root of $\xi_i$ for $1\leq i \leq s$. Since each of the monomials of $\mathbb{A}_{\mathbf{q}}$ acts on $N$ non trivially, therefore all the $\alpha_k$'s are nonzero and so $\underline{\alpha}\in \left(\mathbb{K}^*\right)^n$.
\par Any $\mathbb{A}_{\mathbf{q}}$-module homomorphism must map $X_iX_{i+s}$-eigenvectors of $M(\underline{\alpha})$ to $X_iX_{i+s}$-eigenvectors of $N$
with the same eigenvalue. Define a linear transformation \[\phi:M(\underline{\alpha}) \longrightarrow N\]
by specifying the image of basis vectors of $M(\underline{\alpha})$ only:
\begin{equation}\label{linear}
\phi\left(e(a_1,\cdots,a_s)\right):=\prod \limits_{i=1}^s \alpha_i^{-a_i}~vX_s^{a_s}\cdots X_1^{a_1}.
\end{equation}
To prove $\phi$ is an $\mathbb{A}_{\mathbf{q}}$-module homomorphism, it suffices to check that 
\begin{eqnarray}
 &\phi(e(a_1,\cdots,a_s)X_i)=\phi(e(a_1,\cdots,a_s))X_i,\ \ \ \ \ \ \ \ \ \ \ \ &\forall \ \ \ 1\leq i\leq s \label{h1}\\ 
 &\phi(e(a_1,\cdots,a_s)X_{i+s})=\phi(e(a_1,\cdots,a_s))X_{i+s},\ \ \ \ \ \ &\forall \ \ \ 1\leq i\leq s \label{h2}\\ 
 &\phi(e(a_1,\cdots,a_s)X_{2s+j})=\phi(e(a_1,\cdots,a_s))X_{2s+j},\ \ &\forall \ \ \ 1\leq j\leq n-2s \label{h3}
\end{eqnarray} 
The relations (\ref{h1})-(\ref{h3}) can be easily verified using the linear map $\phi$ in (\ref{linear}) and $\mathbb{A}_{\mathbf{q}}$-module $M(\underline{\alpha})$. To verify the only relation (\ref{h2}), the following computation will be very useful : for $1\leq i \leq s$ and $0\leq a_i\leq k_i-1$,
\[
vX_s^{a_s}\cdots X_1^{a_1}X_{i+s}=
\begin{cases}
\alpha_{i+s}\left(q^{h_i}\right)^{a_i-1}~vX_s^{a_s}\cdots X_i^{a_i-1}\cdots  X_1^{a_1},&\text{when} \ a_i > \ 0.\\
\xi_i^{-1}\alpha_{i+s}\left(q^{h_i}\right)^{{k_i}-1}~vX_s^{a_s}\cdots X_i^{{k_i}-1}\cdots  X_1^{a_1},&\text{when}  \ a_i = \ 0.
\end{cases}
\]
\par Hence $\phi$ is a non zero $\mathbb{A}_{\mathbf{q}}$-module homomorphism. Since $M(\underline{\alpha})$ and $N$ are both simple $\mathbb{A}_{\mathbf{q}}$-module, by Schur's lemma $\phi$ is an isomorphism.
\end{proof}
The above theorem provides an opportunity for classification of simple $\mathcal{O}_{\mathbf{q}}(\mathbb{K}^n)$-modules in terms of parameters.
\begin{coro} \label{surjective}
There is a surjective map $\Psi$ from $(\mathbb{K}^{*})^n$ onto the set of all simple $\mathcal{O}_{\mathbf{q}}(\mathbb{K}^n)$-modules such that $\dime_\mathbb{K}\Psi\left(\underline{\alpha}\right)=\prod \limits_{i=1}^s k_i~,~~\forall~ \underline{\alpha} \in (\mathbb{K}^{*})^n$.
 \end{coro}

\par Let $\underline{\alpha}$, $\underline{\beta}$ be two elements in $\left(\mathbb{K}^*\right)^n$ such that $M(\underline{\alpha})$ and $M(\underline{\beta})$ are isomorphic as $\mathbb{A}_{\mathbf{q}}$-module. As in the space $M(\underline{\alpha})$,
\[e(a_1,\cdots,a_s)=\prod \limits_{i=1}^s\alpha_i^{-a_i}~e(0,\cdots,0)X^{a_1}_1\cdots X^{a_s}_s,\]
the $\mathbb{A}_{\mathbf{q}}$-module isomorphism $\psi:M(\underline{\alpha}) \rightarrow M(\underline{\beta})$ can be uniquely determined by the only image of $e(0,\cdots,0)$, i.e., say
\begin{equation}\label{sum}
\psi\left(e(0,\cdots,0)\right)= \sum_{\text{finite}}\lambda_t~e\left(b^{(t)}_1,\cdots,b^{(t)}_s\right),    
\end{equation}
 for $\lambda_t \in \mathbb{K}^*$. Suppose there exist two non-zero coefficients, say, $\lambda_u,\lambda_v$.
We can choose the smallest index $r$ such that $b_r^{(u)}\neq b_r^{(v)}$ in ${\mathbb{Z}}/{k_r\mathbb{Z}}$. 
Now from the equality \[\psi\left(e(0,\cdots,0)X_rX_{r+s}\right)=\psi\left(e(0,\cdots,0)\right)X_rX_{r+s}\] we must have:
\[
\alpha_{r+s}=\beta_{r+s}\left(q^{h_r}\right)^{b^{(u)}_r}=\beta_{r+s}\left(q^{h_r}\right)^{b^{(v)}_r}\implies k_r\mid (b^{(u)}_r-b^{(v)}_r),
\]
which is a contradiction. Hence it follows that the image in  (\ref{sum}) is of the form \[\psi\left(e(0,\cdots,0)\right)= \lambda~e\left(r_1,\cdots,r_s\right),\] for some $\lambda \in \mathbb{K}^*$ and $r_i\in {\mathbb{Z}}/{k_i\mathbb{Z}}$. Also using this isomorphism $\psi$, we obtain a relation between $\underline{\alpha}$ and $\underline{\beta}$ as mentioned below in the equation (\ref{iso}).

\begin{theo}\label{inj}
Let $\underline{\alpha},\underline{\beta}\in (\mathbb{K}^*)^n$ such that 
\begin{equation}\label{iso}
  \begin{rcases}
 \alpha^{k_i}_i=\beta^{k_i}_i, & \forall\ \ \  1\leq i\leq s\\
\alpha_{i+s}=\beta_{i+s}\left(q^{h_i}\right)^{r_i}~,\ r_i\in {\mathbb{Z}}/{k_i\mathbb{Z}}, & \forall\ \ \  1\leq i\leq s\\
\alpha_{2s+j}=\beta_{2s+j}, & \forall \ \ \ 1\leq j\leq {n-2s}
  \end{rcases}
\end{equation}
Then there is an $\mathbb{A}_{\mathbf{q}}$-module isomorphism $\phi:M(\underline{\alpha})\rightarrow M(\underline{\beta})$ given by 
\[\phi \left(e(a_1,\cdots,a_s)\right)=\prod\limits_{i=1}^{s}\left(\alpha^{-1}_i\beta_i\right)^{a_i}e\left(a_1+r_1,\cdots,a_s+r_s\right).\]
\end{theo}
Thus the above discussion defines an equivalence relation on ${(\mathbb{K}^*)}^n$ so that each equivalence class uniquely determines a simple $\mathcal{O}_{\mathbf{q}}(\mathbb{K}^n)$-module up to isomorphism. 

\section{Conclusion}
From representation theoretic perspective, we have completed a fundamental goal for the multiparameter quantum affine space $\mathcal{O}_{\mathbf{q}}(\mathbb{K}^n)$ by finding and classifying irreducible representations. Hence it follows that all the primitive ideals of the algebra $\mathcal{O}_{\mathbf{q}}(\mathbb{K}^n)$ are annihilators of $~\Psi\left(\underline{\alpha}\right)$ for $\underline{\alpha} \in (\mathbb{K^{*})}^n$. Another way of recognizing primitive ideals of the algebra $\mathcal{O}_{\mathbf{q}}(\mathbb{K}^n)$ in the root of unity case, without first finding irreducible representations, have been studied by K. R. Goodearl and E. S. Letzter  in \cite{glt}.  
\section*{Acknowledgements}
We are very grateful to Dr. Ashish Gupta for setting up the problem, going through the paper and supervising it. We would also like to thank the National Board for Higher Mathematics, Department of Atomic Energy, India for funding our research.

\end{document}